\documentclass[a4paper]{article}
\usepackage{hyperref}
\usepackage{graphicx}
\usepackage{mathrsfs}
\usepackage{enumerate}
\usepackage{amsxtra,amssymb,latexsym, amscd,amsthm,amsmath}
\usepackage{indentfirst}
\usepackage{color}
\usepackage[utf8]{inputenc}
\usepackage[mathscr]{eucal}
\usepackage{amsfonts}
\usepackage{graphics}
\usepackage{multirow}
\usepackage{array}
\usepackage{subfigure}
\usepackage{cite}
\usepackage{wrapfig}
\usepackage{tikz}
\usepackage{diagbox}

\usepackage{pgfplots}
\pgfplotsset{compat=1.15}
\usepackage{mathrsfs}
\usetikzlibrary{arrows}

\newcommand{\footremember}[2]{%
    \footnote{#2}
    \newcounter{#1}
    \setcounter{#1}{\value{footnote}}%
}
\newcommand{\footrecall}[1]{%
    \footnotemark[\value{#1}]%
} 

\def\R{{\mathbb R}}

\def\N{{\mathbb N}}

\DeclareMathOperator{\rank}{rank}

\DeclareMathOperator{\bit}{bit}

\DeclareMathOperator{\sing}{sing}
\DeclareMathOperator{\reg}{reg}
\DeclareMathOperator{\KKT}{KKT}
\DeclareMathOperator{\FJ}{FJ}

\newtheorem{theorem}{\bf Theorem}
\newtheorem{lemma}{\bf Lemma}
\newtheorem{algorithm}{\bf Algorithm}
\newtheorem{example}{\bf Example}
\newtheorem{corollary}{\bf Corollary}
\newtheorem{remark}{\bf Remark}

\providecommand{\keywords}[1]
{
  \textbf{\textbf{Keywords:}} #1
}
\begin{document}
\definecolor{qqzzff}{rgb}{0,0.6,1}
\definecolor{ududff}{rgb}{0.30196078431372547,0.30196078431372547,1}
\definecolor{xdxdff}{rgb}{0.49019607843137253,0.49019607843137253,1}
\definecolor{ffzzqq}{rgb}{1,0.6,0}
\definecolor{qqzzqq}{rgb}{0,0.6,0}
\definecolor{ffqqqq}{rgb}{1,0,0}
\definecolor{uuuuuu}{rgb}{0.26666666666666666,0.26666666666666666,0.26666666666666666}
\newcommand{\vi}[1]{\textcolor{blue}{#1}}
\newif\ifcomment
\commentfalse
\commenttrue
\newcommand{\comment}[3]{%
\ifcomment%
	{\color{#1}\bfseries\sffamily#3%
	}%
	\marginpar{\textcolor{#1}{\hspace{3em}\bfseries\sffamily #2}}%
	\else%
	\fi%
}

\newcommand{\mapr}[1]{{{\color{blue}#1}}}
\newcommand{\revise}[1]{{{\color{blue}#1}}}
\newcommand{\victor}[1]{{{\color{red}#1}}}

\title{Sums of squares representations on singular loci}

\author{%
Ngoc Hoang Anh Mai\footremember{1}{University of Konstanz; D-78464 Konstanz, Germany.}\qquad
Victor Magron\footrecall{1} \footremember{2}{Universit\'e de Toulouse; LAAS; F-31400 Toulouse, France.}
  }

\maketitle

\begin{abstract}
The problem of characterizing a real polynomial $f$ as a sum of squares of polynomials on a real algebraic variety $V$ dates back to the pioneering work of Hilbert in  \cite{hilbert1888darstellung}.
In this paper, we investigate this problem with a focus on cases where the real zeros of $f$ on $V$ are singular points of $V$. 
By using optimality conditions and irreducible decomposition, we provide a positive answer to the following essential question of polynomial optimization: \emph{Are there always exact semidefinite programs to compute the minimum value attained by a given polynomial over a given real algebraic variety?} 
Our answer implies that Lasserre's hierarchy, which is known as a bridge between convex and non-convex programs with algebraic structures, has finite convergence not only in the generic case but also in the general case. 
As a result, we constructively prove that \emph{each hyperbolic program is equivalent to a semidefinite program}.
\end{abstract}
\keywords{sum of squares; Nichtnegativstellensatz; gradient ideal; singular locus;  polynomial optimization; Karush--Kuhn--Tucker conditions; semidefinite programming}
\tableofcontents
\section{Introduction}
\paragraph{Semidefinite programs, Positivstellens\"atze, and Polynomial optimization.}
Semidefinite programming, a subfield of convex optimization, was developed in the early 1960s by Bellman and Fan \cite{bellman1963systems}. Its goal is to minimize a linear objective function over the intersection of the cone of positive semidefinite matrices with an affine space. One of the significant applications of semidefinite programming is to relax a class of non-convex programs with algebraic structures, which are known as polynomial optimization problems. These problems have objective and constraint functions that are all polynomials.

In 2001, Lasserre introduced an appropriate hierarchy of semidefinite programs in \cite{lasserre2001global} that returns a sequence of values approximately converging to the optimal value of a given polynomial optimization problem. To do this, he utilized Positivstellens\"atze, which are representations of polynomials positive on a basic semi-algebraic set, a set defined by a system of finitely many polynomial inequalities. For instance, Putinar's Positivstellensatz \cite{putinar1993positive}, which guarantees the convergence of Lasserre's hierarchy, says that each polynomial positive on a compact basic semi-algebraic set $S$, satisfying the so-called Archimedean condition, can be decomposed as a linear combination of polynomials defining $S$, with weights that are sums of squares of polynomials.
\paragraph{Exact semidefinite programs and Nichtnegativstellens\"atze.}
We aim to develop exact semidefinite programs for determining the optimal value of a given polynomial optimization problem. 
In this context, the term ``exact semidefinite program'' means that, given the objective and constraint polynomials of a polynomial optimization problem, we can algorithmically construct a semidefinite program whose optimal value is precisely equal to that of the original problem. 
Obtaining such programs is highly significant since it enables us to use convex optimization to solve non-convex problems. 
Moreover, it is typically challenging to bound the gap between the convex relaxation and the original problem, as is the case with Goemans-Williamson's algorithm for Max-Cut \cite{goemans1995improved}.

To provide such exact semidefinite programs, we need to build Nichtnegativstellens\"atze that are the representations of polynomials non-negative on a basic semi-algebraic set.
Marshall showed in \cite{marshall2006representations,marshall2009representations} 
a Nichtnegativstellensatz having the same form as Putinar's under the so-called \emph{boundary Hessian conditions}.
Based on this, Nie proves in \cite{nie2014optimality} that Lasserre's hierarchy has finite convergence under generic assumptions related to second-order optimality conditions.
The works of Marshall and Nie rely on the local-global principle stated by Scheiderer in \cite{scheiderer2000sums,scheiderer2003sums,scheiderer2006sums}.
There Scheiderer proved the non-strict extension of Schm\"udgen's Positivstellensatz \cite{schmudgen1991thek}.
The former says that every polynomial $f$ positive on a compact basic semi-algebraic set $S$ is a linear combination of products of polynomials defining $S$ with weights that are sums of squares of polynomials.
In his proof, Scheiderer needs a finiteness assumption on the real zeros of $f$ on $S$.
We refer the readers to \cite{burgdorf2012pure} for some extensions by Burgdorf, Scheiderer, and Schweighofer that allow us to remove the finiteness assumption on the real zeros of $f$ on $S$ but still maintain the compactness of $S$.
Under strong regularity assumption, Nie obtains in his work\cite{nie2013exact} exact semidefinite programs when using the Jacobian of the input polynomials.

\paragraph{Nichtnegativstellens\"atze based on optimality conditions.}
We consider a semi-algebraic set $S$ as a real manifold. 
A point $a$ in $S$ is called a singular (resp. regular) point of $S$ if the tangent space of $S$ at $a$ has smaller (resp. the same) dimension than $S$. 
The singular (resp. regular) locus of $S$ is the set of all singular (resp. regular) points of $S$.

Given a polynomial $f$ non-negative on a basic semi-algebraic set $S$, Demmel, Nie, and Powers provide in \cite{demmel2007representations} a way to obtain a Nichtnegativstellensatz on the intersection $S\cap V^{\KKT}$.
Here $V^{\KKT}$ represents for the variety defined by the Karush--Kuhn--Tucker conditions for the minimization of $f$ on $S$.
Using the Fritz-John conditions, the first author extends this approach in \cite{mai2022exact} to the case where the image of the singular locus $S^{\sing}$ of $S$ under $f$, denoted by $f(S^{\sing})$, is finite.
He also indicates in \cite{mai2022exact} several examples where the previous Nichtnegativstellens\"atze are inapplicable since the real zeros of $f$ on $S$ are in $S^{\sing}$.
In \cite{mai2022complexity2} the first author provides the degree bounds for these Nichtnegativstellens\"atze and analyzes the convergence rate for their application to polynomial optimization.
We emphasize that the Karush--Kuhn--Tucker conditions and the finiteness assumption on $f(S^{\sing})$ belong to the \emph{generic} case.
By generic, we mean that the properties hold in a Zariski open set in the space of the coefficients of the input polynomials with given degrees.
The remaining case, corresponding to a set $f(S^{\sing})$ of  infinite cardinality, has not been tackled so far.

Regarding the above two types of first-order optimality conditions, the Karush--Kuhn--Tucker conditions and the Fritz-John conditions play roles in characterizing the real zeros of polynomial $f$ on a semi-algebraic set $S$.
Moreover, it is not hard to constructively prove the sums of squares-based representations of $f$ on the intersection of $S$ with the varieties defined by these optimality conditions.
To do so, we use the fact that $f$ is constant on each connected component of $V^{\KKT}$, the variety defined by the Karush--Kuhn--Tucker conditions associated to the minimization of $f$ on $S$.
This property also holds for $V^{\FJ}$, the variety defined by the Fritz-John conditions, under the finiteness assumption on $f(S^{\sing})$. 
The difference here is that $V^{\KKT}$ contains the real zeros of $f$ that are regular points of $S$ while $V^{\FJ}$ includes all real zeros of $f$ on $S$ even if they are singular points of $S$.

\paragraph{Contribution.} In the paper, we aim to provide some representations (with degree bounds) of a polynomial $f$ non-negative on a basic semi-algebraic set $S$ in the general case. 

On one hand, we convert $S$ to a real algebraic variety $V$ (the common real zeros of a system of polynomials) in higher-dimensional real space and consider the representations of $f$ on $V$.
Set $U:=V$.
We then decompose $U$ into finitely many irreducible components $U_j$.
If $U_j$ has zero dimension, then $f$ is identical to a sum of squares of polynomials on $U_j$.
Otherwise, $f$ is identical to a sum of squares of polynomials on $U_j^{\KKT}$, the subvariety of $U_j$ defined by the Karush--Kuhn--Tucker conditions associated to the minimization of $f$ on $U_j$.
In this case, we update $U:=U_j^{\sing}$ with $U_j^{\sing}$ being the singular locus of $U_j$ and repeat the above process.
It will terminate after a finite number of steps because the singular locus $U_j^{\sing}$ has lower dimension than $U_j$. 

On the other hand, we provide the degree bounds of our sums of squares-based representations for $f$ on $U_j^{\KKT}$ (or zero-dimensional $U_j$).
Consequently, we analyze the convergence rate of Lasserre's hierarchy applied to minimizing a polynomial $p$ on each $U_j^{\KKT}$ (or zero-dimensional $U_j$). 
We prove that Lasserre's hierarchy on some  $U_j^{\KKT}$ (or zero-dimensional $U_j$) has finite convergence to the minimum value $p^\star$ attained by $p$ on $V$.
Based on this and the algebraic structure of hyperbolic cones, we demonstrate that each hyperbolic program can be written as a semidefinite program.

\paragraph{Hyperbolic polynomials and hyperbolic programming.} For interested readers, hyperbolic programming, introduced by G\"uler in \cite{guler1997hyperbolic}, is concerned with optimizing a linear objective function over the intersection of an affine space with the so-called hyperbolic cone constrained by a hyperbolic polynomial.
Here a real polynomial $p$ is hyperbolic with respect to a given vector $e$ if the univariate polynomial $t\mapsto p(te-a)$ has only real roots for all vectors $a$.
We refer the readers to  \cite{helton2007linear,
netzer2012polynomials,
netzer2013determinantal,
kummer2019spectrahedral,
kummer2017determinantal,
kummer2015hyperbolic,
schweighofer2019spectrahedral,
naldi2018symbolic
} 
for some recent works on hyperbolic polynomials and hyperbolic programming.
Their studies have focused on attacking the generalized Lax conjecture, whose direct consequence is that every hyperbolic program is equivalent to a semidefinite program (see, e.g., \cite[Remark 1.7]{amini2019spectrahedrality}).
 In our paper, we prove the latter. Readers might wonder if the equivalence of the two programs indicates the coincidence of their feasible sets, leading to a solution for the generalized Lax conjecture. 
From our point of view, the generalized Lax conjecture is about fixed cones used in different optimization problems. It is much different from saying that a hyperbolic optimization problem with a given objective function $f$ can be solved by a specific semidefinite program depending on the hyperbolic program associated with $f$.
\paragraph{Previous works.} 
In \cite{mai2022nichtnegativstellensatz}, the first author provides a sum of squares-based representation of a polynomial $f$ nonnegative on a real algebraic variety $V$ in the case where the regular locus $V^{\reg}$ of $V$ is dense in $V$.
The key idea is to consider $V$ as the image of a regular variety $W$ under a morphism $\varphi$, defined by a vector of polynomials.
This is done thanks to Hironaka's resolution of singularities \cite{hironaka1964resolution,hironaka1964resolution2}.
Then we use Demmel--Nie--Powers' Nichtnegativstellensatz to get the representation of $f\circ \varphi$ on $W_{\KKT}$, the variety defined by the Karush--Kuhn--Tucker conditions for minimizing $f\circ \varphi$ on $W$.
However, this method is not always applicable to cases where $V^{\reg}$ is not dense in $V$ (e.g., $V$ is the Whitney or Cartan umbrella).
Compared to this, there is no matter with our method in such challenging cases.

Bucero and Mourrain present in \cite{bucero2013exact} a technique  to obtain  exact semidefinite programs for minimizing a polynomial $p$ over a semi-algebraic set $S$ in the generic case. 
Firstly, they consider the equivalent minimization of $p$ on a new semi-algebraic set $S_{\FJ}$, the intersection of $S$ with the projection of the variety defined by the Fritz-John conditions for minimizing $p$ on $S$.
Secondly, they decompose $S_{\FJ}$ into two semi-algebraic sets $S_{\KKT}$ and $S_{\sing}$, then consider the minimization of $p$ on each of these two sets.
Here $S_{\KKT}$ is the intersection of $S$ with the projection of the variety defined by the Karush--Kuhn--Tucker conditions for minimizing $p$ on $S$.
The set $S_{\sing}$ contains the singular locus of $S$.
They obtain exact semidefinite relaxations for  minimizing $p$ on $S_{\KKT}$ thanks to the work of Demmel, Nie, and Powers \cite{demmel2007representations}.
The minimization of $p$ on $S_{\sing}$ is solved recursively.
If $S_{\sing}$ is zero-dimensional, they obtain exact semidefinite programs for  minimizing $p$ on the finite set $S_{\sing}$.
The case of positive-dimensional $S_{\sing}$ has not been handled entirely in their paper.

Regarding the exactness property, we emphasize that our representations in this paper are sum-of-squares exactness.
In his work \cite{baldi2020exact} with Baldi, the third author considers moment exactness.
It is stronger than sum-of-squares exactness in practice since it yields the minimizers and tests exactness for polynomial optimization. 
Moreover, it is shown in \cite[Theorem 4.14]{baldi2020exact} that if the polar variety defined by the product of minors (and related to the projection of the Karush--Kuhn--Tucker variety) is finite, then moment exactness also holds.

\paragraph{Motivation.}
Our method described in this paper has potential applications to mathematical programs with complementarity constraints (see, e.g., \cite{albrecht2017mathematical,mai2022symbolic2}). 
These optimization problems are challenging because the complementarity constraints typically violate all standard constraint qualifications, making it difficult to find a solution. Constraint qualifications are known to be sufficient conditions for the Karush--Kuhn--Tucker conditions in nonlinear programming. 
Therefore, classical methods that rely on the Karush--Kuhn--Tucker conditions to solve mathematical programs with complementarity constraints are limited. 
In contrast, our method does not require any constraint qualifications

\paragraph{Organization.}
We organize the paper as follows: 
Section \ref{sec:Preliminaries} presents some preliminaries from real algebraic geometry needed to prove our main results.
Section \ref{sec:decomp.sing} is to provide an algorithm that allows us to obtain zero-dimensional and positive-dimensional subvarieties containing singular points of a real algebraic variety.
Section \ref{sec:represent.sing} is to state some
Nichtnegativstellens\"atze on these subvarieties. 
Section \ref{sec:higher-order.opt} is to build higher-order optimality conditions that characterize global minimizers for a polynomial optimization problem.
Section \ref{sec:POP.general} is to present the main algorithm that enables us to obtain exact semidefinite programs for polynomial optimization problems with global minimizers.
Section \ref{sec:hyperbolic} shows how to convert a hyperbolic program into a semidefinite program.

We give some interesting examples to illustrate our results.
We perform some calculations on these examples in Julia 1.7.1 with the software Oscar \cite{OSCAR}. 
The codes for them are available in the link:  \url{https://github.com/maihoanganh/SingularSOS}.
\section{Preliminaries}
\label{sec:Preliminaries}
\subsection{Real algebraic varieties}
Let $\R[x]$ denote the ring of polynomials with real coefficients in the vector of variables $x=(x_1,\dots,x_n)$.
Let $\R_r[x]$ denote the linear space of polynomials in $\R[x]$ of degree at most $r$.

Given $h_1,\dots,h_l$ in $\R[x]$, we denote by $V(h)$ the (real) algebraic variety in $\R^n$ defined by the vector $h=(h_1,\dots,h_l)$, i.e.,
\begin{equation}
V(h):=\{x\in \R^n\,:\,h_j(x)=0\,,\,j=1,\dots,l\}\,.
\end{equation}
In this case, $h_1,\dots,h_l$ are called the polynomials defining $V(h)$. 

Given $h_1,\dots,h_l\in\R[x]$, let $I(h)[x]$ be the ideal generated by $h=(h_1,\dots,h_l)$, i.e.,
\begin{equation}
    I(h)[x]:= \sum_{j=1}^l h_j \R[x]\,.
\end{equation}
The real radical of an ideal $I(h)[x]$, denoted by $\sqrt[\R]{I(h)[x]}$, is defined as
\begin{equation}
{\sqrt[\R]{I(h)[x]}}:=\{f\in\R[x]\,:\,\exists m\in \N\,:\,-f^{2m}\in\Sigma^2[x]+I(h)[x]\}\,.
\end{equation}
Krivine--Stengle's Nichtnegativstellensatz \cite{krivine1964anneaux} imply that
\begin{equation}\label{eq:real.radi}
\sqrt[\R]{I(h)[x]}:=\{p\in\R[x]\,:\,p=0\text{ on }V(h)\}\,.
\end{equation}
We say that $I(h)[x]$ is real radical if $I(h)[x]=\sqrt[\R]{I(h)[x]}$.

Given an algebraic variety $V$ in $\R^n$, we denote by $I(V)$ the vanishing ideal of $V$, i.e.,
\begin{equation}
I(V):=\{p\in\R[x]\,:\,p(x)=0\,,\,\forall x\in V\}\,.
\end{equation}
Note that $I(V)$ is a real radical ideal.
If $V$ is defined by $h=(h_1,\dots,h_l)$ with $h_j\in\R[x]$, then  $I(V)=\sqrt[\R]{I(h)[x]}$.
To compute the generators of $I(V)$ with given polynomials $h_1,\dots,h_l$ defining $V$, we can use, e.g., Becker--Neuhaus' method in \cite{becker1993computation,neuhaus1998computation}. 
Thereby the degrees of the generators of $I(V)$ are bounded from above by $d^{2^{\mathcal O(n^2)}}$ if the degrees of $h_j$s are at most $d$.
In addition, computing the generators of $I(V)$ can be done via the kernel of the moment matrices (or annihilator of moment sequences) as in, e.g., Baldi--Mourrain's method \cite{baldi2021computing}.
Then it could be interesting to get bounds on the degree of the generators returned by this method.

Given an algebraic variety $V$ in $\R^n$, we say that $V$ is irreducible if there do not exist two proper subvarieties $V_1$, $V_2$ in $\R^n$ such that $V = V_1 \cup V_2$.

We recall the following result stated in \cite[Section 4.6, Theorem 4]{cox2013ideals} and \cite[Lemma 4.9]{sheffer2022polynomial} about the structure of algebraic varieties:
\begin{lemma}\label{lem:decomp.irr}
Let $V$ be an algebraic variety in $\R^n$.
Then $V$ can be written as a finite union
\begin{equation}
V = V_1 \cup \dots \cup V_r\,,
\end{equation}
where each $V_j$ is an irreducible algebraic variety in $\R^n$ such that $V_j\not\subset V_t$ if $j\ne t$.
Moreover, if $V$ is defined by polynomials in $\R_d[x]$, then $r$ is bounded by a constant depending solely on $n$ and $d$.
\end{lemma}
We call $V_1,\dots,V_r$ in Lemma \ref{lem:decomp.irr} irreducible components of $V$.
There are several algorithms for finding the polynomials defining $V_j$, $j=1,\dots,r$,  with given polynomials defining $V$. 
For instance, Gianni, Trager, and Zacharias suggest an irreducible decomposition (also called primary decomposition) in \cite{gianni1988grobner}.
\subsection{Regular and singular loci}
Given $h=(h_1,\dots,h_l)$ with $h_j\in\R[x]$, we denote by $J(h)$ the Jacobian matrix associated with $h$, i.e.,
\begin{equation}
J(h)(x):=\left(\frac{\partial h_j}{\partial x_t}(x)\right)_{1\le t\le n,1\le j\le l}\,.
\end{equation}
We define the dimension of an algebraic variety $V$ in $\R^n$, denoted by $\dim(V)$, to be the highest dimension at points at which $V$ is a real submanifold.
For convenience we assume $\dim(\emptyset)=0$ in this paper.

Let $V$ be a algebraic variety in $\R^n$ of dimension $d$.
Let $h_1,\dots,h_l\in\R[x]$ be the generators of $I(V)$.
Set $h=(h_1,\dots,h_l)$.
We say that $a\in V$ is a regular point of $V$ if the Jacobian matrix $J(h)(a)$ has rank $n-d$.
Here the rank of a matrix $M$ with real coefficients is the largest integer $r$ such that all $(r+1)\times(r+1)$ minors of $M$ vanish.
The set $V^{\reg}$ of all regular point of $V$ is called the regular locus of $V$.
Let $V^{\sing}=V\backslash V^{\reg}$.
We say that $a\in V$ is a singular point of $V$ if $a\in V^{\sing}$, i.e., the Jacobian matrix $J(h)(a)$ has rank smaller than $n-d$.
The set $V^{\sing}$ is called the singular locus of $V$.
The tangent space of $V$ at $a\in V$, denoted by
$T_a(V)$, is the linear subspace of $\R^n$ given by
\begin{equation}\label{eq:tangent}
T_a(V):=\{u\in\R^n\,:\,J(h)(a)^\top u=0\}\,.
\end{equation}
It is not hard to prove that $a\in V$ is a regular point of $V$ iff $\dim(T_a(V))=d=\dim(V)$.

Given $h=(h_1,\dots,h_l)$ with $h_j\in\R[x]$, denote by $m_t(h)(x)$ the vector of $t\times t$ minors of the Jacobian matrix $J(h)(x)$.
Then $m_t(h)$ has length $\binom{n}{t}\times \binom{l}{t}$.
Each entry of $m_t(h)$ is in $\R[x]$ and has degree at most $t\times \max_j\deg(h_j)$.

The following lemma states some basic properties of singular loci:
\begin{lemma}\label{lem:property.sing}
Let $V$ be an algebraic variety in $\R^n$ of dimension $d$.
Let $h_1,\dots,h_l\in\R[x]$ be the generators of $I(V)$.
Set $h=(h_1,\dots,h_l)$.
Then the following conditions hold:
\begin{enumerate}
\item The singular locus of $V$ is an algebraic variety in $\R^n$ defined by $(h,m_{n-d}(h))$.
\item The singular locus of $V$ has lower dimension than $V$.
\end{enumerate}
\end{lemma}
\begin{proof}
The first statement is proved similarly to \cite[Section 6.2]{smith2004invitation}.
The proof of the second one can be found in \cite[Theorem 4.8]{sheffer2022polynomial} (see also \cite[Section 12.4.2]{sitharam2018handbook}).
\end{proof}

\subsection{First-order optimality conditions}
Given $h_0,h_1,\dots,h_l$ in $\R[x]$, consider the following polynomial optimization problem:
\begin{equation}\label{eq:pop}
    h^\star:=\inf\limits_{x\in V(h)} h_0(x)\,,
\end{equation}
where $V(h)$ is the algebraic variety in $\R^n$ defined by $h=(h_1,\dots,h_l)$.
\begin{remark}\label{rem:convert}
The more general form 
\begin{equation}
\begin{array}{rl}
\inf\limits_{y\in\R^r}&f(y)\\
\text{s.t.}&g_j(y)\ge 0\,,\,j=1,\dots,l\,,
\end{array}
\end{equation}
can be written as an instance of \eqref{eq:pop} by setting $x=(y,z)$, $h_0(x)=f(y)$, $h_j(x)=g_j(y)-z_j^2$.
\end{remark}
Given $p\in\R[x]$, we denote by $\nabla p$ the gradient of $p$, i.e., $\nabla p=(\frac{\partial p}{\partial x_1},\dots,\frac{\partial p}{\partial x_n})$.
We recall the Karush--Kuhn--Tucker conditions in the following lemma:
\begin{lemma}\label{lem:KKT}
Let $h_0$ in $\R[x]$. 
Let $V$ be an algebraic variety in $\R^n$ of dimension $d>0$. 
Let $h_1,\dots,h_l\in\R[x]$ be the generators of $I(V)$.
Let $x^\star$ be a local minimizer for problem \eqref{eq:pop} with $h=(h_1,\dots,h_l)$.
Assume that $J(h)(x^\star)$ has rank $n-d$.
Then the Karush--Kuhn--Tucker conditions hold for problem \eqref{eq:pop} at $x^\star$, i.e.,
\begin{equation}\label{eq:KKT}
    \begin{cases}
    		\exists (\lambda_1^\star,\dots,\lambda_l^\star)\in \R^{l}\,:\\
    		h_j(x^\star)=0\,,\,j=1,\dots,l\,,\\
          \nabla h_0(x^\star)=\sum_{j=1}^l \lambda_j^\star \nabla h_j(x^\star)\,.
    \end{cases}
\end{equation}
\end{lemma}
\begin{proof}
By assumption, $x^\star$ is a regular point of the manifold $V\subset \R^n$.
Then there exists a diffeomorphism $\Phi:U\to V$ for some open set $U\subset \R^d$ such that $x^\star=\Phi(t^\star)$ for some $t^\star\in U$.
The differential of $\Phi$ at $t\in U$ is defined by the linear mapping $D \Phi_t: \R^d\to T_{\Phi(t)}V$, $u\mapsto J(\Phi)(t)^\top u$, where $T_{a}V$ is the tangent space of $V$ at $a\in V$ (defined as in \eqref{eq:tangent}).
Since $x^\star$ is a regular points of $V$,  $D \Phi_{t^\star}$ is bijective.
From this, we get $\rank (J(\Phi)(t^\star)) =d$, which gives the null space of $J(\Phi)(t^\star)$, denoted by $	J(\Phi)(t^\star)^\perp$, has dimension $n-d$ thanks to the rank--nullity theorem.
By assumption, $t^\star$ is a local minimizer of $h_0\circ \Phi$ on $U$.
It implies that $0=\nabla (h_0\circ \Phi)(t^\star)=J(\Phi)(t^\star)\times \nabla h_0(x^\star)$, which gives $\nabla h_0(x^\star)$ is in $J(\Phi)(t^\star)^\perp$.
In addition, for $j=1,\dots,l$, $(h_j\circ \Phi)(t)=0$ for all $t\in U$. 
Take the gradient in $t$, we obtain $J(\Phi)(t)\times \nabla h_j(\Phi(t))=0$, for all $t\in U$, for $j=1,\dots,l$.
It implies that $\nabla h_j(\Phi(t))$, $j=1,\dots,l$, are in the null space of $J(\Phi)(t)$, for all $t\in U$.
By assumption, the linear span of $\nabla h_j(x^\star)$, $j=1,\dots,l$, has dimension $n-d$.
Since $J(\Phi)(t^\star)^\perp$ has dimension $n-d$,  $J(\Phi)(t^\star)^\perp$ is the linear span of $\nabla h_j(x^\star)$, $j=1,\dots,l$.
Hence the result follows since $\nabla h_0(x^\star)$ is in $J(\Phi)(t^\star)^\perp$.
\end{proof}
\begin{remark}
As shown in Freund's lecture note \cite[Theorem 11]{freund2004optimality}, the Karush--Kuhn--Tucker conditions \eqref{eq:KKT} hold for problem \eqref{eq:pop} at $x^\star$ when the linear independence constraint qualification is satisfied, i.e., the gradients $\nabla h_j(x^\star)$, $j=1,\dots,l$, are linearly independent in $\R^n$, which is equivalent to that $J(h)(x^\star)$ has rank $l$.
For comparison purposes, we make a weaker assumption in Lemma \ref{lem:KKT} that $J(h)(x^\star)$ has rank $n-\dim(V)$. 
Similarly to \cite[Exercise 17 b, page 495]{cox2013ideals}, we obtain $l\ge n-\dim(V)$ in general.
Note that the twisted cubic $V = \{(t,t^2,t^3)\in \R^3\,:\,t\in \R\}$ is a one-dimensional variety defined by $x_2-x_1^2$ and $x_3-x_1^3$, but the ideal $I(V)$ is generated by the vector of three polynomials $h=(x_1x_3-x_2^2,x_2-x^2_1,x_3-x_1 x_2)$. Thus it holds that $l=3>2=n-\dim(V)$ in this example.
\end{remark}

To prove that the largest rank assumption of $J(h)(x^\star)$ in Lemma \ref{lem:KKT} cannot be removed, consider the following example:
\begin{example}\label{exam:KKT.not.hold}
Let $n=2$, $h_0=x_1$, and $h=x_1^3-x_2^2$.
Then $x^\star=(0,0)$ is the unique global minimizer for problem \eqref{eq:pop}.
Moreover, the Karush--Kuhn--Tucker conditions do not hold for problem \eqref{eq:pop} at $x^\star$.
Indeed, for any $\lambda\in\R$, we get
\begin{equation}
\nabla h_0(x^\star)-\lambda\nabla h(x^\star)=\begin{bmatrix}
1\\
0
\end{bmatrix}-\lambda\begin{bmatrix}
3x_1^{*2}\\
-2x_2^\star
\end{bmatrix}=\begin{bmatrix}
1\\
0
\end{bmatrix}\ne 0\,.
\end{equation}
Note that $V(h)$ has dimension $d=1$, and $J(h)=\nabla h=\begin{bmatrix}
3x_1^{2}\\
-2x_2
\end{bmatrix}$.
It is not hard to check that $J(h)(x^\star)$ has rank  $0< 1=n-d$.
\end{example}

Given $\bar h:=(h_0,h)$ with $h:=(h_1,\dots,h_l)$ and  $h_j\in\R[x]$, we denote by $\bar h_{\KKT}$ the vector of polynomials in $\R[x, \lambda]$ associated with the Karush--Kuhn--Tucker conditions defined by
\begin{equation}
\bar h_{\KKT}:=(h,\nabla h_0-\sum_{j=1}^l \lambda_j \nabla h_j)\,,
\end{equation}
where $\lambda=(\lambda_1,\dots,\lambda_l)$.
The condition \eqref{eq:KKT} can be written as $(x^\star,\lambda^\star)\in V(\bar h_{\KKT})$ for some $\lambda^\star\in\R^l$.
\subsection{Semi-algebraic set}

Given $g=(g_1,\dots,g_m)$ with $g_j\in\R[x]$, we denote by $S(g)$ the basic semi-algebraic set associated with $g$, i.e.,
\begin{equation}
    S(g):=\{x\in\R^n\,:\,g_j(x)\ge 0\,,\,j=1,\dots,m\}\,.
\end{equation}

A semi-algebraic subset of $\R^n$ is a subset of the following form
\begin{equation}\label{eq:def.semi.set}
\bigcup_{i=1}^t\bigcap_{j=1}^{r_i}\{x\in\R^n\,:\,f_{ij}(x)*_{ij}0\}\,,
\end{equation}
where $f_{ij}\in\R[x]$ and $*_{ij}\in\{>,=\}$.
Note that \eqref{eq:def.semi.set} is the union of finitely many basic semi-algebraic sets.

Given two semi-algebraic sets $A\subset \R^n$ and $B\subset \R^m$, we say that a mapping $f : A \to B$ is semi-algebraic if its graph $\{(x,f(x))\,:\,x\in A\}$ is a semi-algebraic set in $\R^{n+m}$.
A semi-algebraic subset $A\subset \R^n$ is said to be semi-algebraically path connected if for every $x,y$ in $A$, there exists a continuous semi-algebraic mapping $\phi:[0,1] \to A$ such that $\phi(0) = x$ and $\phi(1) = y$.

The following lemma can be found in \cite[Proposition 1.6.2 (ii)]{pham2016genericity}:
\begin{lemma}\label{lem:composit.semi-al}
Compositions of semi-algebraic maps are semi-algebraic.
\end{lemma}

The following lemma is given in \cite[Theorem 1.8.1]{pham2016genericity}:
\begin{lemma}\label{lem:semial.func.anal}
Let $f:(a, b)\to \R$ be a semi-algebraic
function. Then there are $a = a_0 < a_1 < \dots < a_s < a_{s+1} = b$ such that, for each $i = 0,\dots, s$, the restriction $f|_{(a_i,a_{i+1})}$ is analytic.
\end{lemma}
The following lemma follows from the mean value theorem:
\begin{lemma}\label{lem:mean.val}
Let $f:[0,1]\to \R$ be a continuous piecewise-differentiable function, i.e., there exist $0=a_1<\dots<a_r=1$ such that $f$ is continuous and $f$ is differentiable on each open interval $(a_i,a_{i+1})$.
Assume that $f$ has zero subgradient.
Then $f(0)=f(1)$.
\end{lemma}
\begin{proof}
By using the mean value theorem on each open interval $(a_i,a_{i+1})$, we get $f(a_i)=f(a_{i+1})$.
Hence $f(0)=f(a_1)=\dots=f(a_{r})=f(1)$ yields the result.
\end{proof}

Given $n,d,s\in\N$, we define
\begin{equation}\label{eq:bound.connected}
c(n,d,s):=d(2d-1)^{n+s-1}\,.
\end{equation}

We recall in the following lemma the upper bound on the number of connected components of a basic semi-algebraic set is stated by Coste in \cite[Proposition 4.13]{coste2000introduction}:
\begin{lemma}\label{lem:num.connected}
Let $g_1,\dots,g_m,h_1,\dots,h_l\in\R_d[x]$ with $d\ge 2$. 
The number of (semi-algebraically path) connected components of $S(g)\cap V(h)$ is not greater than $c(n,d,m+l)$.
\end{lemma}
\subsection{Sums of squares}
Denote by $\Sigma^2[x]$ (resp. $\Sigma^2_r[x]$) the cone of sums of squares of polynomials in $\R[x]$ (resp. $\R_r[x]$).
Given $g_1,\dots,g_m\in\R[x]$, let $P_r(g)[x]$ be the truncated preordering of order $r\in\N$ associated with $g=(g_1,\dots,g_m)$, i.e., 
\begin{equation}
    P_r(g)[x]:=\{\sum_{\alpha\in\{0,1\}^m} \sigma_\alpha g^\alpha\,:\,\sigma_\alpha\in\Sigma^2[x]\,,\,\deg(\sigma_\alpha g^\alpha)\le 2r\}\,,
\end{equation}
where $\alpha=(\alpha_1,\dots,\alpha_m)$ and $g^\alpha:=g_1^{\alpha_1}\dots g_m^{\alpha_m}$.
If $m=0$, it holds that $P_r(g)[x]=\Sigma^2_r[x]$.

Given $h_1,\dots,h_l\in\R[x]$, let $I_r(h)[x]$ be the truncated ideal of order $r$ defined by $h$, i.e.,
\begin{equation}
    I_r(h)[x]:= \{\sum_{j=1}^l h_j \psi_j\,:\,\psi_j\in\R[x]\,,\,\deg(h_j \psi_j)\le 2r\}\,.
\end{equation}

We denote by $\bit(d)$ the number of bits of $d\in\N$, i.e.,
\begin{equation}
\bit(d):=
 \begin{cases}
          1 & \text{if } d = 0\,,\\
	    k & \text{if } d \ne 0 \text{ and } 2^{k-1}\le d < 2^k.
         \end{cases}
\end{equation}
Given $n,d,s\in\N$, we define 
\begin{equation}\label{eq:bound.Krivine}
b(n,d,s):=2^{
2^{\left(2^{\max\{2,d\}^{4^{n}}}+s^{2^{n}}\max\{2, d\}^{16^{n}\bit(d)} \right)}}\,.
\end{equation}

We recall the degree bounds for Krivine--Stengle's Nichtnegativstellens\"atze by Lombardi, Perrucci, and Roy \cite{lombardi2020elementary} in the following two lemmas:
\begin{lemma}\label{lem:pos}
Let $g_1,\dots,g_m,h_1,\dots,h_l$ in $\R_d[x]$. 
Assume that $S(g)\cap V(h)=\emptyset$ with $g:=(g_1,\dots,g_m)$ and $h:=(h_1,\dots,h_l)$. 
Set $r=\frac{1}{2}\times b(n,d,m+l+1)$.
Then it holds that $-1 \in P_r(g)[x]+I_r(h)[x]$.
\end{lemma}
\begin{lemma}\label{lem:pos2}
Let $p,g_1,\dots,g_m,h_1,\dots,h_l$ in $\R_d[x]$. 
Assume that $p$ vanishes on $S(g)\cap V(h)$ with $g:=(g_1,\dots,g_m)$ and $h:=(h_1,\dots,h_l)$. 
Set $r:=\frac{1}{2}\times b(n,d,m+l+1)$ and $s:=2\lfloor r/d\rfloor$.
Then it holds that $-p^s \in P_r(g)[x]+ I_r(h)[x]$.
\end{lemma}
\subsection{Nichtnegativstellens\"atze on regular loci}
\label{sec:proof.rep}
Denote by $|\cdot|$ the cardinality of a set and by $\delta_{ij}$ the Kronecker delta function at $(i,j)\in\N^2$.

We state in the following lemma a sums of squares-based representation with degree bound for a polynomial which has finitely many non-negative values on a real algebraic variety:
\begin{lemma}\label{lem:quadra}
Let $h_0,h_1,\dots,h_l$ in $\R_{d}[x]$. 
Assume that $h_0$ is non-negative on $V(h)$  and $h_0(V(h))$ is finite with $h=(h_1,\dots,h_l)$. 
Set $r:=|h_0(V(h))|$ and $u:=\frac{1}{2}\times b(n,d,l+1)$.
Then there exists $\sigma\in \Sigma^2_w[x]$ with $w=\max\{d(r-1),u+d\}$ such that $h_0 - \sigma$ vanishes on $V(h)$.
\end{lemma}
\begin{proof}
Consider the following two cases:
\begin{itemize}
\item Case 1: $r=0$. It is obvious that  $V(h)=\emptyset$.
Lemma \ref{lem:pos} says that $-1 \in \Sigma^2_u[x]+I_u(h)[x]$.
It implies that there exists $q \in P_u(g)[x]$ such that
$-1 = q$ on $V(h)$. 
We write $h_0 = s_1 - s_2$, where $s_1 =(h_0+\frac{1}{2})^2$ and $s_2 = h_0^2+\frac{1}{4}$ are in $\Sigma^2[x]$.
From this we get $h_0 = s_1 + q s_2$ on $V(h)$.
Letting $\sigma = s_1 +qs_2$ gives $\sigma\in \Sigma^2_{u+d}[x]\subset \Sigma^2_{w}[x]$ since $w\ge u+d$.
Thus $h_0-\sigma$ vanishes on $V(h)$.

\item Case 2: $r>0$. By assumption, we can assume that
$h_0(V(h)) = \{t_1 ,\dots, t_r \} \subset [0,\infty)$,
where $t_i\ne t_j$ if $i\ne j$.
For $j=1,\dots,r$, let
$W_j:=V(h,h_0-t_j)$.
Then $W_j$ is a real variety defined by $l+1$ polynomials in $\R_{d}[x]$.
It is clear that $h_0(W_j)=\{t_j\}$.
Define the following polynomials:
\begin{equation}\label{eq:lagrange.pol}
p_j(x):=\prod_{i\ne j}\frac{h_0(x)-t_i}{t_j-t_i}\,,\,j=1,\dots,r\,.
\end{equation}
It is easy to check that $p_j(W_i)=\{\delta_{ji}\}$ and $\deg(p_j)\le d(r-1)$.
Note that $h_0=t_i\ge 0$ on $W_i$, for $i=1,\dots,r$.  
Now letting $\sigma =\sum_{i=1}^r t_i p_i^2$,
we obtain $\sigma\in \Sigma^2_w[x]$ since $\deg(t_i p_i^2)\le 2\deg(p_i)\le 2d(r-1)\le 2w$.
Hence $h_0 - \sigma$ vanishes on $V(h)=W_1\cup\dots\cup W_r$, yielding the result.
\end{itemize}
\end{proof}

The following lemma is similar to \cite[Lemma 3.3]{demmel2007representations} but is proved by using the tools from real algebraic geometry (instead of the ones from complex algebraic geometry):
\begin{lemma}\label{lem:constant.KKT}
Let $h_1,\dots,h_l$ in $\R[x]$.
Let $h_0$ be a polynomial in $\R[x]$.
Let $W$ be a semi-algebraically path connected component of $V(\bar h_{\KKT})$, where $\bar h:=(h_0,\dots,h_l)$. Then $h_0$ is constant on $W$.
\end{lemma}
\begin{proof} 
Recall $\lambda:=(\lambda_1,\dots,\lambda_m)$.
Choose two arbitrary points $(x^{(0)},\lambda^{(0)})$, $(x^{(1)},\lambda^{(1)})$ in $W$. 
We claim that $h_0(x^{(0)}) = h_0(x^{(1)})$.
By assumption, there exists a continuous semi-algebraic mapping $\phi:[0,1]\to W$ defined by $\phi(\tau) = (x(\tau),\lambda(\tau))$ such that $\phi(0) = (x^{(0)},\lambda^{(0)})$ and $\phi(1) = (x^{(1)},\lambda^{(1)})$. 
We claim that $\tau\mapsto h_0(x(\tau))$ is constant on $[0,1]$.
The Lagrangian function
\begin{equation}\label{eq:Lagran.KKT}
    L(x,\lambda) := h_0(x)-\sum_{j=1}^m \lambda_j h_j (x)\,
\end{equation}
is equal to $h_0(x)$ on $V(\bar h_{\KKT})$, which contains $\phi([0,1])$. 
By Lemma \ref{lem:composit.semi-al}, the function $L\circ \phi$ is semi-algebraic.
Moreover, the function $L\circ \phi$ is continuous since $L$ and $\phi$ are continuous.
It implies that $L\circ \phi$ is a continuous piecewise-differentiable function thanks to Lemma \ref{lem:semial.func.anal}.
Note that the function
$L\circ \phi$
has zero subgradient on $[0,1]$.
From Lemma \ref{lem:mean.val}, it follows that $h_0(x (0) )=(L\circ \phi)(0)= (L\circ \phi)(1)= h_0(x (1) )$.
We now obtain $h_0(x^{(0)})$ = $h_0(x^{(1)})$ and hence $h_0$ is constant on $W$.
\end{proof}

Based on the Karush--Kuhn--Tucker conditions, we state in the following theorem the sums of squares-based representation with degree bound for a polynomial non-negative on algebraic varieties:
\begin{theorem}\label{theo:rep.KKT}
Let $h_0,h_1,\dots,h_l$ be polynomial in $\R_d[x]$  with $d\ge 2$.
Assume that $h_0$ is non-negative on $V(h)$ with $h=(h_1,\dots,h_l)$.
Set $\bar h:=(h_0,h)$, $\lambda:=(\lambda_1,\dots,\lambda_l)$ and
\begin{equation}\label{eq:def.w.KKT}
w:=\max\{d\times (c(n+l,d,n+l)-1),\frac{1}{2}\times b(n+l,d,l+n+1)+d\}\,,
\end{equation}
where $c(\cdot)$ and $b(\cdot)$ are defined as in \eqref{eq:bound.connected} and  \eqref{eq:bound.Krivine}, respectively.
Then the following statements hold:
\begin{enumerate}
\item The cardinality of $h_0(V(\bar h_{\KKT}))$ is at most $c(n+l,d+1,n+l)$.
\item There exists $\sigma\in \Sigma^2_w[x, \lambda]$ such that $h_0-\sigma$ vanishes on $V(\bar h_{\KKT})$.
\end{enumerate}
\end{theorem} 
\begin{proof}
Using Lemma \ref{lem:num.connected}, we decompose $V(\bar h_{\KKT})$ into semi-algebraically path connected components:
$Z_1,\dots,Z_s$ with  
\begin{equation}\label{eq:bound.on.s.KKT}
s\le c(n+l,d,n+l)\,,
\end{equation}
since each entry of $\bar h_{\KKT}$ has degree at most $d\ge 2$.
Accordingly Lemma \ref{lem:constant.KKT} shows that $h_0$ is constant on each $Z_i$.
Thus $h_0(V(\bar h_{\KKT}))$ is finite.
Set $r=|h_0(V(\bar h_{\KKT}))|$.
From \eqref{eq:bound.on.s.KKT}, we get 
\begin{equation}\label{eq:ineq.KKT}
r\le s\le c(n+l,d,n+l)\,.
\end{equation}
Set 
\begin{equation}
u:=\frac{1}{2}\times b(n+l,d,l+n+1)\,.
\end{equation}
By using Lemma \ref{lem:quadra}, there exists $\sigma\in \Sigma^2_\xi[x,\lambda]$ with $\xi=\max\{d(r-1),u+d\}$ such that $h_0 - \sigma$ vanishes on $V(\bar h_{\KKT})$.
By \eqref{eq:ineq.KKT} and \eqref{eq:def.w.KKT}, we get $\xi\le w$, and hence $\sigma\in \Sigma^2_w[x,\lambda]$. 
\end{proof}
\begin{remark}
Regarding the representation in Theorem \ref{theo:rep.KKT}, it could also be interesting to stay in the $x$-space instead of lifting to the $(x,\lambda)$-space.
To do this, we need the projections of the Karush--Kuhn--Tucker varieties onto the $x$-space and handle the representations similar to the ones of Bucero and Murrain in \cite{bucero2013exact}.
Another open/difficult question is to get better (e.g., double exponential) bounds for the representation of polynomials.
\end{remark}
We state in the following theorem the degree bound for Nie's Nichtnegativstellensatz \cite{nie2013polynomial}, which is the sums of squares-based representation for a polynomial non-negative on zero-dimensional algebraic varieties:
\begin{theorem}\label{theo:rep.finite}
Let $h_0,h_1,\dots,h_l$ be polynomial in $\R_d[x]$  with $d\ge 2$.
Assume that $V(h)$ with $h=(h_1,\dots,h_l)$ has zero dimension and $h_0$ is non-negative on $V(h)$.
Set
\begin{equation}\label{eq:def.w.finite}
w:=\max\{d\times (c(n,d,l)-1),\frac{1}{2}\times b(n,d,l+1)+d\}\,,
\end{equation}
where $c(\cdot)$ and $b(\cdot)$ are defined as in \eqref{eq:bound.connected} and  \eqref{eq:bound.Krivine}, respectively.
Then the following statements hold:
\begin{enumerate}
\item The cardinality of $h_0(V(h))$ is at most $c(n,d,l)$.
\item There exists $\sigma\in \Sigma^2_w[x]$ such that $h_0-\sigma$ vanishes on $V(h)$.
\end{enumerate}
\end{theorem} 
\begin{proof}
Since $V(h)$ has zero-dimension, $V(h)$ has a finite number $s$ of distinct points.
By Lemma \ref{lem:num.connected}, we get $s\le c(n,d,l)$ since each entry of $h$ has degree at most $d\ge 2$.
Then $h_0$ is constant on each $Z_i$ which yields $h_0(V(\bar h_{\KKT}))$ is finite.
Set $r=|h_0(V(h))|$.
It implies that $r\le s\le c(n,d,l)$.
Set $u:=\frac{1}{2}\times b(n,d,l+1)$.
By using Lemma \ref{lem:quadra}, there exists $\sigma\in \Sigma^2_\xi[x]$ with $\xi=\max\{d(r-1),u+d\}$ such that $h_0 - \sigma$ vanishes on $V(h)$.
It is not hard to prove that $\xi\le w$, and hence $\sigma\in \Sigma^2_w[x]$. 
\end{proof}

\subsection{Exact polynomial optimization in the generic case}

We recall some preliminaries  of the Moment-SOS relaxations originally developed by Lasserre in \cite{lasserre2001global}.
Given $d\in\N$, let $\N^n_d:=\{\alpha\in\N^n\,:\,\sum_{j=1}^n \alpha_j\le d\}$.
Given $d\in\N$, we denote by $v_d$ the vector of monomials in $x$ of degree at most $d$, i.e., $v_d=(x^\alpha)_{\alpha\in\N^n_d}$ with $x^\alpha:=x_1^{\alpha_1}\dots x_n^{\alpha_n}$.
For each $p\in\R_d[x]$, we write $p=c(p)^\top v_d=\sum_{\alpha\in\N^n_d}p_\alpha x^\alpha$, where $c(p)$ is denoted by the vector of coefficient of $p$, i.e., $c(p)=(p_\alpha)_{\alpha\in\N^n_d}$ with $p_\alpha\in\R$.
Given $A\in\R^{r\times r}$ being symmetric, we say that $A$ is positive semidefinite, denoted by $A\succeq 0$, if every eigenvalue of $A$ is non-negative.

Given $y=(y_\alpha)_{\alpha\in\N^n}\subset \R$, let $L_y:\R[x]\to\R$ be the Riesz linear functional defined by $L_y(p)=\sum_{\alpha\in\N^n} p_\alpha y_\alpha$ for every $p\in\R[x]$.
Given $d\in\N$, $p\in\R[x]$ and $y=(y_\alpha)_{\alpha\in\N^n}\subset \R$, let $M_d(y)$ be the moment matrix of order $d$ defined by $(y_{\alpha+\beta})_{\alpha,\beta\in\N^n_d}$.

The following lemma shows the connection between sums of squares and semidefinite programming (see, e.g., \cite[Proposition 2.1]{lasserre2015introduction}):
\begin{lemma}\label{lem:sdp.sos}
Let $\sigma\in\R[x]$ and $d\in\N$ such that $2d\ge\deg(\sigma)$.
Then $\sigma\in\Sigma^2[x]$ iff there exists $G\succeq 0$ such that $\sigma=v_d^\top Gv_d$.
\end{lemma}
Given $k\in\N$ and $h_0,h_1,\dots,h_l\in\R[x]$, consider the following primal-dual semidefinite programs associated with $\bar h=(h_0,h_1,\dots,h_l)$:
\begin{equation}\label{eq:mom.relax}
\begin{array}{rl}
\tau_k(\bar h):=\inf\limits_y& L_y(h_0)\\
\text{s.t} &M_k(y)\succeq 0\,,\\
&M_{k-r_t}(h_ty)=0\,,\,t=1,\dots,l\,,\,y_0=1\,,
\end{array}
\end{equation}
\begin{equation}\label{eq:sos.relax}
\begin{array}{rl}
\rho_k(\bar h):=\sup\limits_{\xi,G,u_t} & \xi\\
\text{s.t} & G\succeq 0\,,\\
&h_0-\xi=v_k^\top Gv_k+\sum_{t=1}^l h_tu_t^\top v_{2(k-r_t)}\,,\\
\end{array}
\end{equation}
where $r_t=\lceil \deg(h_t)/2\rceil$.
Using Lemma \ref{lem:sdp.sos}, we obtain 
\begin{equation}\label{eq:equi.sos}
\rho_k(\bar h)=\sup_{\xi\in\R}\{ \xi\,:\,h_0-\xi\in \Sigma_k[x]+I_k(h)[x]\}\,.
\end{equation}
Primal-dual semidefinite programs \eqref{eq:mom.relax}-\eqref{eq:sos.relax} is known as the Moment-SOS relaxations of order $k$ for problem \eqref{eq:pop}.

We state in the following lemma some recent results involving the Moment-SOS relaxations:
\begin{lemma}\label{lem:mom.sos}
Let $h_0,h_1,\dots,h_l\in\R_{d}[x]$. 
Let $h^\star$ be as in \eqref{eq:pop}  with $h=(h_1,\dots,h_l)$. 
Set $\bar h=(h_0,h)$.
Assume that $h^\star>-\infty$.
Then the following statements hold:
\begin{enumerate}
\item For every $k\in\N$, $\tau_k(\bar h)\le \tau_{k+1}(\bar h)$ and $\rho_k(\bar h)\le \rho_{k+1}(\bar h)$.
\item For every $k\in\N$, $\rho_k(\bar h)\le \tau_{k}(\bar h)\le h^\star$.
\item If there exists $q\in \Sigma_w[x]$ with $2w\ge d$ such that $h_0-h^\star-q$ vanishes on $V(h)$, then  $\rho_r(\bar h)=h^\star$ with $r=\frac{1}{2}\times b(n,2w,l+1)+d$, where  $b(\cdot)$ is defined as in   \eqref{eq:bound.Krivine}.
\end{enumerate}
\end{lemma}
\begin{proof}
The proofs of the first two statements are trivial.
Let us use Nie's technique in \cite[Proof of Theorem 1.1]{nie2014optimality} to prove the third statement. (It allows us to avoid using real radical but still obtain exact semidefinite programs thanks to Krivine--Stengle's Nichtnegativstellens\"atze.)
Consider the following two cases:
\begin{itemize}
\item Case 1: $V(h)= \emptyset$. 
Then we get $h^\star=\infty$.
Set $s=\frac{1}{2}\times b(n,d,l+1)$.
Then Lemma \ref{lem:pos2} says that $-1 \in \Sigma^2_s[x]+I_s(h)[x]$.
For all $\xi\ge 0$, it holds that
\begin{equation}
h_0-\xi=(1+\frac{h_0}{4})^2-(\xi+(1-\frac{h_0}{4})^2)\in \Sigma_{s+d}[x]+I_{s+d}(h)[x]\,.
\end{equation}
Since $r\ge s+d$, it implies that for all $\xi\ge 0$, $h_0-\xi\in \Sigma_{r}[x]+I_{r}(h)[x]$, which yields that $\xi$ is a feasible solution for \eqref{eq:equi.sos} of the value $\rho_r(\bar h)$.
Thus we obtain that $\rho_r(\bar h)=\infty= h^\star$.
\item Case 2: $V(h)\ne \emptyset$. 
Then we get $h^\star<\infty$.
Set $u=h_0-h^\star-q$.
By assumption,we get $u\in\R_{2w}[x]$ and $u=0$ on $V(h)$. 
Set $s=2\lfloor (r-d)/(2w)\rfloor$.
From this, Lemma \ref{lem:pos2} says that there exist  $\sigma \in \Sigma_{r-d}^2[x]$ such that $u^{2s} + \sigma \in I_{r-d}(h)[x]$.
Let $c=\frac{1}{2s}$. 
Then it holds that $1+t+ct^{2s}\in\Sigma^2_s[t]$.
Thus for all $\varepsilon>0$, we have
\begin{equation}
\begin{array}{rl}
h_0-h^\star+\varepsilon&=q + \varepsilon(1+\frac{u}{\varepsilon}+c\left(\frac{u}{\varepsilon}\right)^{2s})-c\varepsilon^{1-2s}(u^{2s} + \sigma ) +c\varepsilon^{1-2s}\sigma\\
&\in \Sigma^2_{r-d}[x]+I_{r-d}(h)[x]\subset \Sigma^2_{r}[x]+I_{r}(h)[x]\,.
\end{array}
\end{equation}
Then we for all $\varepsilon>0$, $h^\star-\varepsilon$ is a feasible solution of \eqref{eq:equi.sos} of the value $\rho_r(\bar h)$.
It gives $\rho_r(\bar h)\ge h^\star-\varepsilon$, for all $\varepsilon>0$, and, in consequence, we get $\rho_r(\bar h)\ge h^\star$.
Using the second statement, we obtain that $\rho_r(\bar h)= h^\star$, yielding the third statement.
\end{itemize}
\end{proof}

We apply Theorem \ref{theo:rep.KKT} for polynomial optimization as follows:
\begin{theorem}\label{theo:pop.KKT}
Let $h_0,h_1,\dots,h_l\in\R_d[x]$. 
Let $h^\star$ be as in problem \eqref{eq:pop} with $h=(h_1,\dots,h_l)$.
Assume that problem \eqref{eq:pop} has a global minimizer at which the Karush--Kuhn--Tucker conditions hold for this problem.
Let $w$ be as in \eqref{eq:def.w.KKT}.
Set 
\begin{equation}\label{eq:r.formula}
r:=\frac{1}{2}\times b(n+l,2w,l+n+1)+d\,.
\end{equation}
Then $\rho_r(h_0,\bar h_{\KKT})=h^\star$, where $\bar h:=(h_0,h)$ and $b(\cdot)$ is defined as in  \eqref{eq:bound.Krivine}.
\end{theorem}
\begin{proof}
By assumption, there exists $(x^\star,\lambda^\star)\in V(\bar h_{\KKT})$ such that $x^\star$ is a global minimizer of \eqref{eq:pop}.
It implies that
\begin{equation}
\begin{array}{rl}
h^\star:=\min\limits_{x,\lambda}& h_0(x)\\
\text{s.t.}& (x,\lambda)\in V(\bar h_{\KKT})\,,
\end{array}
\end{equation}
By assumption, Theorem \ref{theo:rep.KKT} yields that there exists $q\in \Sigma^2_w(g)[x,\lambda]$ such that $h_0-h^\star-q$ vanishes on $V(\bar h_{\KKT})$.
Applying the third statement of Lemma \ref{lem:mom.sos}, we obtain the conclusion.
\end{proof}
\begin{remark}\label{rem:not.attained}
Let $h=(h_1,\dots,h_l)$ with $h_j\in\R[x]$ and let $h^\star$ be as in \eqref{eq:pop}.
Assume that $h^\star$ is finite but is not attained i.e., $h_0(x)-h^\star>0$ for all $x\in V(h)$. (For instance, we can take (i) $h_0=x_1$ and $h=(x_1x_2^2-1)$ or (ii) $h_0=(x_1x_2-1)^2 + x_1^2$ and $h=(0)$.)
By Theorem \ref{theo:rep.KKT}, the set $h_0(V(\bar h_{\KKT}))-h^\star$ (with $\bar h:=(h_0,h)$) has a finite number of values but does not have zero value.
It is because of 
\begin{equation}
0\notin h_0(V(h))-h^\star\supset h_0(V(\bar h_{\KKT}))-h^\star\,.
\end{equation}
It implies that $\inf (h_0(V(\bar h_{\KKT}))-h^\star)=\delta>0$, so that $\inf h_0(V(\bar h_{\KKT}))=h^\star+\delta$.
Note that $\delta=\infty$ iff $h_0(V(\bar h_{\KKT}))=\emptyset$.
Thus we obtain $\rho_r(h_0,\bar h_{\KKT})=h^\star+\delta>h^\star$, where $w$ is as in \eqref{eq:def.w.KKT} and $r$ is as in \eqref{eq:r.formula}.
To address this attainability issue, see Remark \ref{rem:attain}.
\end{remark}
We present the application of Theorem \ref{theo:rep.finite} to polynomial optimization:
\begin{theorem}\label{theo:pop.finite}
Let $h_0,h_1,\dots,h_l\in\R_d[x]$. 
Let $h^\star$ be as in problem \eqref{eq:pop} with $h=(h_1,\dots,h_l)$.
Assume that $V(h)$ has zero-dimension.
Let $w$ be as in \eqref{eq:def.w.finite}.
Set 
\begin{equation}\label{eq:r.formula.finite}
r:=\frac{1}{2}\times b(n,2w,l+1)+d\,.
\end{equation}
Then $\rho_r(h_0,h)=h^\star$, where $b(\cdot)$ is defined as in  \eqref{eq:bound.Krivine}.
\end{theorem}
\begin{proof}
By assumption, Theorem \ref{theo:rep.finite} yields that there exists $q\in \Sigma^2_w[x]$ such that $h_0-h^\star-q$ vanishes on $V(h)$.
Applying the third statement of Lemma \ref{lem:mom.sos}, we obtain the conclusion.
\end{proof}

\section{Decomposition of singular loci}
\label{sec:decomp.sing}
The following algorithm allows us to obtain the zero-dimensional and positive-dimensional subvarieties containing singular points of a real algebraic variety:
\begin{algorithm}\label{alg:reduce.dim.sing} 
Decomposition of singular loci.
\begin{itemize}
\item Input: Algebraic variety $V$ in $\R^n$  defined by polynomials in $\R[x]$.
\item Output: Sets $A$ and $B$ of subvarieties of $V$.
\end{itemize}
\begin{enumerate}
\item Set $U:=V$, $A:=\emptyset$  and $B:=\emptyset$.
\item Let $U_1,\dots,U_r$ be the irreducible components of $U$.
\item Set $A:=A\cup\{U_1,\dots,U_r\}$.
\item For $j=1,\dots,r$, do:
\begin{enumerate}
\item If $U_j$ is zero-dimensional, then set $A:=A\backslash \{U_j\}$ and $B:=B\cup\{U_j\}$.
\item Otherwise, let $U_j^{\sing}$ be the singular locus of $U_j$, set $U:=U_j^{\sing}$ and run again Steps 2, 3, 4.
\end{enumerate}
\end{enumerate}
\end{algorithm}
\begin{remark}
In order to verify if $U_j$ is zero-dimensional as in Step 4 (a) of Algorithm \ref{alg:reduce.dim.sing}, we refer the readers to the work of Lairez and Safey El Din in \cite{lairez2021computing}.
Moreover, their method also allows us to compute the dimension of a general real algebraic variety.
\end{remark}
\begin{remark}
In Step 4 (b) of Algorithm \ref{alg:reduce.dim.sing}, we implicitly compute generators $h$ of the real radical $I(U_j)$ to get the singular locus $U_j^{\sing}$ as defined by $(h,m_{n-d}(h))$ (see the first statement of Lemma \ref{lem:property.sing}).
In particular, if $U_j^{\sing}$ is defined by $(h,m_{n-d}(h))$ for some vector of polynomials $h$ defining $U_j$ and having multiplicities (e.g., $h=(q^2)$  for some $q\in\R[x]$), then we have $U_j^{\sing}=U_j$, and Algorithm \ref{alg:reduce.dim.sing} loops forever.
Other approaches based on deflation techniques and the so-called Thom--Boardman stratification could avoid this issue. 
They allow us to apply directly to the equations and derivatives (see, e.g., the work of Hauenstein and Wampler \cite{hauenstein2013isosingular}).
These methods are probably not very efficient but might give better bounds for polynomials defining singular loci than relying on the computation of real radical based on Becker--Neuhaus' method in \cite{becker1993computation,neuhaus1998computation}.
\end{remark}
The following lemma provides some properties of the output of Algorithm \ref{alg:reduce.dim.sing}.
\begin{lemma}\label{lem:well.alg}
Let $A$ and $B$ be the output of Algorithm \ref{alg:reduce.dim.sing}.
The following statements hold:
\begin{enumerate}
\item  $A$ and $B$ are the sets of finite numbers of algebraic varieties in $\R^n$ defined by polynomials in $\R[x]$.
\item Each variety in $A$ is irreducible and  positive-dimensional.
\item Each variety in $B$ is irreducible and zero-dimensional.
\item \if{$A\cap B=\emptyset$ and}\fi $V=\bigcup\limits_{W\in A\cup B}W$.
\end{enumerate}
\end{lemma}
\begin{proof}
Let us prove the first statement.
By Lemma \ref{lem:decomp.irr} and the first statement of Lemma \ref{lem:property.sing}, $U_j$ and $U_j^{\sing}$ in Step 4 (b) of Algorithm \ref{alg:reduce.dim.sing} are algebraic varieties in $\R^n$ defined by polynomials in $\R[x]$, then so are the elements of $A$ and $B$.
To prove that $A$ and $B$ are finite, we only need to show that Algorithm \ref{alg:reduce.dim.sing} terminates after a finite number of steps.
Assume  by contradiction that Algorithm \ref{alg:reduce.dim.sing} never terminates.
Lemma \ref{lem:decomp.irr} says that the number $r$ in Step 2 of Algorithm \ref{alg:reduce.dim.sing} is finite.
From this, Step 4 (b) produces an infinite sequence of irreducible varieties $(W_t)_{t\in\N}$ such that $W_0$ is an irreducible component of $V$ and
$W_{t+1}$ is an irreducible component of the singular locus of $W_{t}$.
It is impossible since the second statement of Lemma \ref{lem:property.sing} gives
$\dim(V)\ge\dim(W_{t})> \dim(W_{t+1})\ge 0$ for all $t\in\N$.
The last three statements are due to Step 1 and Step 4 (a) of Algorithm \ref{alg:reduce.dim.sing}.
\end{proof}

We now give the following example, which is based on \cite[Example 18]{mai2022exact}:
\begin{example}\label{exam:union.2.rays}
Let $n=5$ and $V=V(h)$ with $h=(x_{1}^3 - x_{4}^2,(x_2-x_3)^3-x_5^2,x_1(x_2-x_3))$ be as the input of Algorithm \ref{alg:reduce.dim.sing}.
In Step 2 of  Algorithm \ref{alg:reduce.dim.sing}, we decompose $V$ into irreducible components: $U_j=V(h^{(j)})$, $j=1,2$, where $h^{(1)}=(x_{5},x_{2} - x_{3},x_{1}^3 - x_{4}^2)$ and $h^{(2)}=(x_{4}, (x_{2}- x_{3})^3 - x_{5}^2, x_{1})$. 
(Both $U_1$ and $U_2$ have dimension two.)
Let $U_j^{\sing}$ be the singular locus of $U_j$, $j=1,2$.
Using the first statement of Lemma \ref{lem:property.sing}, we obtain the vector of polynomials $h^{(1)}_{\sing}=(x_{1}, x_{5}, x_{4}, x_{2} - x_{3})$ defining $U_1^{\sing}$. 
It is clear that $U_1^{\sing}$ is an irreducible variety of dimension one and the singular locus of $U_1^{\sing}$ is empty.
Similarly, we obtain the same vector of polynomials $h^{(2)}_{\sing}=(x_{1}, x_{5}, x_{4}, x_{2} - x_{3})$ defining $U_2^{\sing}$, so we get $U_2^{\sing}=U_1^{\sing}$.
Step 4 (a) of Algorithm \ref{alg:reduce.dim.sing}  yields $A=\{U_1,U_2,U_1^{\sing}\}$ and $B=\{\emptyset\}$.
\end{example}

\section{Nichtnegativstellens\"atze on singular loci}
\label{sec:represent.sing}
We provide in the following theorem the sums of squares-based representations on subvarieties returned by Algorithm \ref{alg:reduce.dim.sing}:
\begin{theorem}\label{theo:rep.nonneg}
Let $q_0,h_1,\dots,h_l$ in $\R[x]$. 
Assume that $q_0$ is non-negative on $V(h)$ with $h:=(h_1,\dots,h_l)$.
Let $A$ and $B$ be the output of Algorithm \ref{alg:reduce.dim.sing} with input $V=V(h)$.
Let $c(\cdot)$ and $b(\cdot)$ be as in \eqref{eq:bound.connected} and  \eqref{eq:bound.Krivine}, respectively.
Then the following statements hold:
\begin{enumerate}
\item For each variety $U$ in $A$ defined by $q=(q_1,\dots,q_r)$ with $q_j\in\R[x]$, then
\begin{enumerate}
\item the cardinality of $q_0(V(\bar q_{\KKT}))$ with $\bar q=(q_0,q)$ is at most 
\begin{equation}\label{eq:def.w.KKT2}
c(n+r,d+1,n+r)\,,
\end{equation}
where $d:=\max\{\deg(q_j)\,:\,j=0,\dots,r\}$;
\item  there exists $\sigma\in \Sigma_w[x,\lambda]$ with $\lambda=(\lambda_1,\dots,\lambda_l)$ and
\begin{equation}\label{eq:def.w.FJ2}
w:=\max\{d\times (c(n+r,d,n+r)-1),\frac{1}{2}\times b(n+r,d,r+n+1)+d\}
\end{equation}
 such that $q_0-\sigma$ vanishes on $V(\bar q_{\KKT})$.
\end{enumerate}

\item For each variety $U$ in $B$ defined by $q=(q_1,\dots,q_r)$ with $q_j \in\R[x]$, then 
\begin{enumerate}
\item the cardinality of $q_0(V(q))$ is at most $c(n,d,r)$,
where $d:=\max\{\deg(q_j)\,:\,j=0,\dots,r\}$;
\item  there exists $\sigma\in \Sigma_w[x]$ with
\begin{equation}
w:=\max\{d\times (c(n,d,r)-1),\frac{1}{2}\times b(n,d,r+1)+d\}
\end{equation}
such that $q_0-\sigma$ vanishes on $V(q)$.
\end{enumerate}
\end{enumerate}
\end{theorem}
\begin{proof}
To prove the first and second statements, we apply Theorems \ref{theo:rep.KKT} and \ref{theo:rep.finite}, respectively.
Note that for each $U$ in $A$, $p_0$ is non-negative on $U\cap \R^n$ since $q_0$ is non-negative on  $V\cap \R^n\supset U\cap \R^n$.
Moreover, the second statement of Lemma \ref{lem:well.alg} yields that each $U$ in $B$ has zero-dimension.
\end{proof}
\begin{example}\label{exam:rep}
We use the same notation as in Example \ref{exam:union.2.rays}. 
Let $q_0=x_1+x_2-x_3$. 
Then $q_0$ is non-negative on $V(h)$.
Assume by contradiction that the non-strict extension of Schm\"udgen's Positivstellensatz is applicable to $q_0$ on $V(h)$, i.e., $q_0\in \Sigma^2_d[x]+I_d(h)[x]$ for some $d\in\N$. 
Letting $x_2=x_3=x_4=x_5=0$, we get $x_1=\sigma+x_1^3\psi$ for some $\sigma\in\Sigma^2[x_1]$ and $\psi\in \R[x_1]$.
It implies that $x_1$ divides $\sigma$ so that $\sigma=x_1^2\sigma_1$ for some $\sigma_1\in\Sigma^2[x_1]$.
From this, $x_1=x_1^2\sigma_1+x_1^3\psi$ implies that $1=x_1(\sigma_1+x_1\psi)$, which is impossible when we set $x_1=0$.
It is not hard to check that Theorem \ref{theo:rep.nonneg}.2 (b) is applicable to this example.
Indeed, letting $\lambda:=(\lambda_1,\lambda_2,\lambda_3,\lambda_4)$ and $\sigma:=0$ implies that $\sigma\in \Sigma[x,\lambda]$ and $q_0-\sigma$ vanishes on $U_1^{\sing}=V(h^{(1)}_{\sing})\in A$.
Thus $q_0-\sigma$ vanishes on $V(\bar q_{\KKT})\subset V(h^{(1)}_{\sing})\times \R^{4}$ with $\bar q=(q_0,h^{(1)}_{\sing})$.
\end{example}
\if{
\begin{remark}
Scheiderer's local-global principle in \cite{scheiderer2003sums} states that
a polynomial $f$ is a sum of squares on a real variety if and only if $f$ is a sum of squares in the completed local ring at each of its real zeros 
\end{remark}
}\fi
\section{Higher-order optimality conditions}
\label{sec:higher-order.opt}
We characterize a local minimizer for problem \eqref{eq:pop} with first-order optimality conditions in the following lemma:
\begin{lemma}\label{lem:pos.minimizer}
Let $V$ be an algebraic variety in $\R^n$.
Let $h_1,\dots,h_l\in \R[x]$ be the generators of $I(V)$.
Let $h_0$ be a polynomial in $\R[x]$ such that problem \eqref{eq:pop} with $h=(h_1,\dots,h_l)$ has a local minimizer $x^\star$.
Then one of the following two cases occurs:
\begin{enumerate}
\item The Karush--Kuhn--Tucker conditions hold for problem \eqref{eq:pop} at $x^\star$.
\item The point $x^\star$ is in the singular locus of $V$.
\end{enumerate}
\end{lemma}
\begin{proof}
If the Jacobian matrix $J(h)(x^\star)$ has rank $n-\dim(V)$, Lemma \ref{lem:KKT} yields that the Karush--Kuhn--Tucker conditions hold for problem \eqref{eq:pop} at $x^\star$.
Otherwise, we get $m_{n-d}(h)(x^\star)=0$, and hence $x^\star$ is in the singular locus of $V$. 
Here $m_{r}(h)(x)$ is the vector of the $r\times r$ minors of the Jacobian matrix $J(h)(x)$ associated with $h$.
\end{proof}
The following theorem shows in which varieties returned by Algorithm \ref{alg:reduce.dim.sing} the local minimizers for problem \eqref{eq:pop} belong, and which types of optimality conditions hold for problem \eqref{eq:pop} at these points:
\begin{theorem}\label{theo:high.order.}
Let $h_0,h_1,\dots,h_l$ be polynomials in $\R[x]$ such that problem \eqref{eq:pop} with $h=(h_1,\dots,h_l)$ has a global minimizer $x^\star$.
Let $A$ and $B$ be as the output of Algorithm \ref{alg:reduce.dim.sing} with input $V=V(h)$.
Then one of the following two cases occurs:
\begin{enumerate}
\item There exists a variety $U$ in $A$ with $q_1,\dots,q_r\in\R[x]$ being the generators of $I(U)$  such that with $q=(q_1,\dots,q_r)$, $x^\star$ is a global minimizer for problem
\begin{equation}\label{eq:sub.pop}
\min_{x\in V(q)}h_0(x)
\end{equation}
and the Karush--Kuhn--Tucker conditions hold for problem \eqref{eq:sub.pop} at $x^\star$.
\item There exists a variety $U$ in $B$  defined by $q=(q_1,\dots,q_r)$ with $q_j\in\R[x]$ such that $x^\star$ is a global minimizer for problem \eqref{eq:sub.pop}.
\end{enumerate}
\end{theorem}
\begin{proof}
By the final statement of Lemma \ref{lem:well.alg}, $x^\star$ belongs to some variety $U$ in $A\cup B$.
By the first statement of Lemma \ref{lem:well.alg}, $U$ is an irreducible variety.
Let $q_1,\dots,q_r\in\R[x]$ be the generators of the vanishing ideal $I(U)$.
Then  $x^\star$ is also the global minimizer for problem \eqref{eq:sub.pop} with $q=(q_1,\dots,q_r)$.
Assume that Karush--Kuhn--Tucker conditions do not hold for problem \eqref{eq:sub.pop} at $x^\star$.
By Lemma \ref{lem:pos.minimizer}, $x^\star$ is in the singular locus $U^{\sing}$ of $U$.
If $U$ has zero dimension then $U$ is in $B$.
Otherwise, we decompose $U^{\sing}$ into irreducible components (Step 4 (b) of Algorithm \ref{alg:reduce.dim.sing}). 
In this case, $x^\star$ belongs to some irreducible component of $U^{\sing}$.
We repeat the above process until obtaining only the zero-dimensional components.
Hence the result follows.
\end{proof}
\begin{remark} 
By Step 4 (b) of Algorithm \ref{alg:reduce.dim.sing} and the first statement of Lemma \ref{lem:property.sing}, each $U$ in $A\cup B$ is an irreducible variety defined by polynomials built from the higher-order partial derivatives of $h_1,\dots,h_l$.
\end{remark}
\section{Exact polynomial optimization in the general case}
\label{sec:POP.general}
The following algorithm enables us to obtain exact semidefinite programs to compute the minimum value $h^\star$ of problem \eqref{eq:pop}.
\begin{algorithm}\label{alg:POP} 
Computing the minimum value $h^\star$ of problem \eqref{eq:pop}.
\begin{itemize}
\item Input: $h_0,h_1,\dots,h_l$ in $\R[x]$.
\item Output: $\bar h^\star$
\end{itemize}
\begin{enumerate}
\item Let $A$ and $B$ be as the output of Algorithm \ref{alg:reduce.dim.sing} with input $V=V(h)$, where $h=(h_1,\dots,h_l)$.
\item Set $T:=\emptyset$.
\item For each variety $U$ in $A$ with $q_1,\dots,q_r\in\R[x]$ being the generators of $I(U)$, do:
\begin{enumerate}
\item Compute $t:=\rho_\xi(h_0,\bar q_{\KKT})$ with $q:=(q_1,\dots,q_r)$, $\bar q=(h_0,q)$, $d:=\max\{\deg(h_0)\,,\,\deg(q_j)\,:\,j=1,\dots,r\}$, $w$ being as in \eqref{eq:def.w.KKT2} and
\begin{equation}\label{eq:xi.KKT}
\xi:=\frac{1}{2}\times b(n+r,2w,r+n+1)+d\,.
\end{equation}
\item Set $T:=T\cup \{t\}$.
\end{enumerate}
\item For each variety $U$ in $B$ defined by $q:=(q_1,\dots,q_r)$ with $q_j\in\R[x]$, do:
\begin{enumerate}
\item Compute $t:=\rho_\xi(h_0,q)$ with $d:=\max\{\deg(h_0)\,,\,\deg(q_j)\,:\,j=1,\dots,r\}$, $w$ being as in \eqref{eq:def.w.FJ2} and
\begin{equation}\label{eq:xi.FJ}
\xi:=\frac{1}{2}\times b(n,2w,r+1)+d\,.
\end{equation}
\item Set $T:=T\cup \{t\}$.
\end{enumerate}
\item Set $\bar h^\star :=\min T$.
\end{enumerate}
\end{algorithm}
\begin{remark}
Step 3 (a) and Step 4 (a) of Algorithm \ref{alg:POP} are equivalent to solving exactly semidefinite programs (of the form \eqref{eq:sos.relax}). 
We refer the readers to the exact algorithm of Henrion, Naldi, and Safey El Din in \cite{henrion2018exact} to solve semidefinite programs.
In addition, the set $T$ in Algorithm \ref{alg:POP} might contain $\infty$ if the polynomial optimization problem associated with some semidefinite program in Step 3 (a) and Step 4 (a) has empty feasible set.
\end{remark}

We state our main result in the following theorem:
\begin{theorem}\label{theo:exact.pop}
Let $h_0,h_1,\dots,h_l\in\R_d[x]$. 
Let $h^\star$ be as in problem \eqref{eq:pop} with $h=(h_1,\dots,h_l)$.
Assume that problem \eqref{eq:pop} has a global minimizer.
Let $\bar h^\star$ be as the output of Algorithm \ref{alg:POP} with input $h_0,h_1,\dots,h_l$.
Then $\bar h^\star=h^\star$.
\end{theorem}
\begin{proof}
Let $x^\star$ be a global minimizer for problem \eqref{eq:pop}.
By Theorem \ref{theo:high.order.}, 
one of the following two cases occurs:
\begin{itemize}
\item Case 1: There exists a variety $U$ in $A$ with $q_1,\dots,q_r\in\R[x]$ being the generators of $I(U)$ such that $x^\star$ is a global minimizer for problem \eqref{eq:sub.pop} with $q=(q_1,\dots,q_r)$
and the Karush--Kuhn--Tucker conditions hold for problem \eqref{eq:sub.pop} at $x^\star$.
By Theorem \ref{theo:pop.KKT}, we get $\rho_\xi(h_0,\bar q_{\KKT})=h^\star$ with $\bar q=(h_0,q)$, $d:=\max\{\deg(h_0)\,,\,\deg(q_j)\,:\,j=0,\dots,r\}$, $w$ being as in \eqref{eq:def.w.KKT2} and $\xi$ being as in \eqref{eq:xi.KKT}.
\item Case 2: There exists a variety $U$ in $B$  defined by $q=(q_1,\dots,q_r)$ with $q_j\in\R[x]$ such that $x^\star$ is a global minimizer for problem \eqref{eq:sub.pop}.
By Theorem \ref{theo:pop.finite}, we get $\rho_\xi(h_0,q)=h^\star$ with $d:=\max\{\deg(h_0)\,,\,\deg(q_j)\,:\,j=0,\dots,r\}$, $w$ being as in \eqref{eq:def.w.FJ2} and $\xi$ being as in \eqref{eq:xi.FJ}.
\end{itemize}
Hence the result follows due to steps 3, 4 and 5 of Algorithm \ref{alg:POP} (as well as Remark \ref{rem:not.attained}).
\end{proof}
Let $\|\cdot\|_2$ denote the $l_2$-norm of a real vector.
Then $\|x\|_2^2=x_1^2+\dots+x_n^2$ is a polynomial in $x$.
\begin{remark}
To find a global minimizer for problem \eqref{eq:pop}, we apply Algorithm \ref{alg:POP}  and the adding-spherical-constraints method in \cite[Section 4.3]{mai2021positivity} as follows:
Compute the minimum value $h^\star$ for problem \eqref{eq:pop} by using Algorithm \ref{alg:POP}.
Setting $\bar h=(h,h_0-h^\star)$ implies that $V(\bar h)$ is the set of global minimizers for problem \eqref{eq:pop}.
We assume that $V(\bar h)$ is non-empty and find a point $x^\star$ in $V(\bar h)$.
Let $(a^{(t)})_{t = 0}^n$ be a finite sequence of points in $\R^n$ such that $a^{(t)} - a^{(0)},\, t = 1,\dots,n$ are linearly independent in $\R^n$.
We use Algorithm \ref{alg:POP} to compute the following values:
\begin{equation}
\begin{array}{l}
\xi _0:= \min \{\|x-a^{(0)}\|_2^2\,:\, x \in V(\bar h) \}\,,\\
\xi_t: = \min \{\|x-a^{(t)}\|_2^2\,:\, x \in V(\bar h,b_0,\dots,b_{t-1})\}\,,\,  t = 1,\dots,n\,.
\end{array}
\end{equation}
where $b_r=\xi _r-\|x-a^{(r)}\|_2^2$.
Then there exists a unique real point $x^\star$ in $V(\bar h,b_0,\dots,b_n)$ which satisfies the non-singular linear system of equations
\begin{equation}\label{eq:linear.equation}
(a^{(t)}-a^{(0)})^\top x^\star = -\frac{1}{2}(\xi_t - \xi_0 +\|a^{(0)}\|_2^2-\|a^{(t)}\|_2^2),\ t = 1,\dots,n\,.
\end{equation}
\end{remark}
We illustrate Algorithm \ref{alg:POP} in the following example, which is based on \cite[Example 18]{mai2022exact}:
\begin{example}\label{exam:pop}
Let $n=5$, $h_0=x_1+x_2$, $h_1=x_{1}^5 - x_{3}^2$, $h_2=x_{2}^5 - x_{4}^2$, and $h_3=-x_{1}x_{2} - x_{5}^2$ be as the input of Algorithm \ref{alg:POP}.
Let $h^\star$ be as in problem \eqref{eq:pop} with $h=(h_1,\dots,h_l)$.
Then it is not hard to prove that $h^\star=0$.
In Step 1 of Algorithm \ref{alg:POP}, we obtain
\begin{equation}
A=\{V(q^{(1)}),V(q^{(2)})\}\quad\text{ and }\quad B=\{V(q^{(3)})\}\,,
\end{equation}
where $q^{(1)}:=(x_{5}, x_{4}, x_{2}, x_{1}^5 - x_{3}^2)$, $q^{(2)}:=(x_{5}, x_{3}, x_{2}^5 - x_{4}^2, x_{1})$ and $q^{(3)}:=(x_{1}, x_{2}, x_{3}, x_{4}, x_{5})$.
For $j=1,2$, letting $\bar q^{(j)}:=(h_0,q^{(j)})$ gives $V(\bar q^{(j)}_{\KKT})=\emptyset$.
Indeed, for any $\lambda\in\R^4$, we get
\begin{equation}
\begin{array}{rl}
&\nabla h_0(x)-\sum_{i=1}^4\lambda_j\nabla q^{(1)}_i(x)\if{=\begin{bmatrix}
1\\
1\\
0\\
0\\
0
\end{bmatrix}-\lambda_1\begin{bmatrix}
0\\
0\\
0\\
0\\
1
\end{bmatrix}-\lambda_2\begin{bmatrix}
0\\
0\\
0\\
1\\
0
\end{bmatrix}-\lambda_3\begin{bmatrix}
0\\
1\\
0\\
0\\
0
\end{bmatrix}-\lambda_4\begin{bmatrix}
5x_1^4\\
0\\
-2x_3\\
0\\
0
\end{bmatrix}}\fi=\begin{bmatrix}
1-5\lambda_4x_1^4\\
1-\lambda_3\\
2\lambda_4x_3\\
-\lambda_2\\
-\lambda_1
\end{bmatrix}\,,\\\\
&\nabla h_0(x)-\sum_{i=1}^4\lambda_j\nabla q^{(2)}_i(x)\if{=\begin{bmatrix}
1\\
1\\
0\\
0\\
0
\end{bmatrix}-\lambda_1\begin{bmatrix}
0\\
0\\
0\\
0\\
1
\end{bmatrix}-\lambda_2\begin{bmatrix}
0\\
0\\
1\\
0\\
0
\end{bmatrix}-\lambda_3\begin{bmatrix}
0\\
5x_2^4\\
0\\
-2x_4\\
0
\end{bmatrix}-\lambda_4\begin{bmatrix}
1\\
0\\
0\\
0\\
0
\end{bmatrix}}\fi=\begin{bmatrix}
1-\lambda_4\\
1-5\lambda_3x_2^4\\
-\lambda_2\\
2\lambda_3x_4\\
-\lambda_1
\end{bmatrix}\,.
\end{array}
\end{equation}
Then $\bar q^{(1)}_{\KKT}(x)=0$ implies that $x_1^5=x_3^2$, $1-5\lambda_4x_1^4=2\lambda_4x_3=0$, which is impossible.
In addition, $\bar q^{(2)}_{\KKT}(x)=0$ implies that $x_2^5=x_4^2$, $1-5\lambda_3x_2^4=2\lambda_3x_4=0$, which is impossible.
For $j=1,2$, since $V(\bar q^{(j)}_{\KKT})=\emptyset$, it holds that $\rho_\xi(h_0,\bar q^{(j)}_{\KKT})=\infty$ for sufficient large $\xi\in\N$. 
Since $V(q^{(3)})=\{0\}$, we get $\rho_\xi(h_0,q^{(3)})=0$  for sufficient large $\xi\in\N$.
Thus we eventually obtain $T=\{\infty,0\}$, yielding $\bar h^\star=\min T=0=h^\star$.
\end{example}
\begin{remark}\label{rem:symbolic}
Instead of using semidefinite programming, we can rely entirely on symbolic computations \cite{mai2022symbolic2} to find the minimum value $h^\star$ of problem \eqref{eq:pop} when this problem has a global minimizer $x^\star$.
(Developed by the first author, the method in \cite{mai2022symbolic2}  depends on the computations of real radical generators and Groebner bases.
It allows us to obtain univariate polynomials defining the Zariski closure of the image of a basic semi-algebraic set $S$ under a polynomial $f$.)
To do this, let $A$ and $B$ be as the output of Algorithm \ref{alg:reduce.dim.sing} with input $V=V(h)$.
By Theorem \ref{theo:rep.nonneg}, for each $U$ in $A$ with $q_1^U,\dots,q_{r_U}^U\in\R[x]$ being the generators of $I(U)$, $h_0(V(\bar q_{\KKT}^U))$  has finitely many values, where $\bar q^U=(h_0,q_1^U,\dots,q_{r_U}^U)$.
Moreover, for each $U$ in $B$, $h_0(U)$ has finitely many values.
By using \cite[Algorithm 1]{mai2022symbolic2}, we can compute all values of $h_0(V(\bar q_{\KKT}^U))$ (resp. $h_0(U)$), for each $U$ in $A$ (resp. $B$).
Let $h^\star_U$ be the smallest value in $h_0(V(\bar q_{\KKT}^U))$ (resp. $h_0(U)$), for each $U$ in $A$ (resp. $B$).
Note that for each $U$ in $A$ (resp. $B$), we have $V(\bar q_{\KKT}^U)\subset V(h)\times \R^{r_U}$ (resp. $U\subset V(h)$), so that $h^\star_U\ge h^\star$.
Theorem \ref{theo:high.order.} says that $x^\star$ belongs to the projection of some $V(\bar q_{\KKT}^U)$ with $U$ in $A$ or $x^\star$ belongs to some $U$ in $B$, which implies that $h^\star=\min \{h^\star_U\,:\,U\in A\cup B\}$.
\end{remark}
\begin{remark}
If we replace Step 2 of Algorithm \ref{alg:reduce.dim.sing} with ``Set $r:=1$ and $U_1:=U$.'', then all results so far still hold.
However, the polynomials defining $U_1$ possibly have very high degrees in this case (as shown in Examples \ref{exam:union.2.rays}, \ref{exam:rep} and \ref{exam:pop}).
\end{remark}
\begin{remark}\label{rem:attain}
The result of Theorem \ref{theo:exact.pop} requires the attainability of the infimum value $h^\star$. 
In \cite{mai2022semi}, the first author proves that every polynomial optimization problem of the form \eqref{eq:pop} with finite infimum value $h^\star$ can be symbolically transformed to an equivalent problem in one-dimensional space with attained optimal value $h^\star$.
To do this, he uses quantifier elimination, and algebraic algorithms that rely on the fundamental theorem of algebra and the greatest common divisor.
Let $d$ be the upper bound on the degrees of $h_j$. 
His symbolic algorithm has complexity $O(d^{O(n)})$ to produce the objective and constraint polynomials of degree at most $d^{O(n)}$ for the equivalent problem.
\end{remark}

The following corollary states the equivalence of polynomial optimization problems and semidefinite programs:
\begin{corollary}\label{coro:equi.pop.sdp}
Let $h_0,h_1,\dots,h_l\in\R_d[x]$. 
Let $h^\star$ be as in problem \eqref{eq:pop} with $h=(h_1,\dots,h_l)$.
Then there is a symbolic algorithm with input $h_0,\dots,h_l$ which produces a finite sequence of semidefinite programs with minimum values $(\zeta_r)_{r=1}^s$ such that \eqref{eq:pop} is equivalent to the semidefinite program with smallest minimum value $\zeta_{\bar r}=\min\limits_{r=1,\dots,s}\zeta_r$.
\end{corollary}
\begin{proof}
If $h^\star\in\{ -\infty,\infty\}$, it is not hard to construct a semidefinite program with infimum value $h^\star$.
If $h^\star$ is finite but not attained, we can transform problem \eqref{eq:pop} to an equivalent polynomial optimization problem with attained minimum value $h^\star$ thanks to Remark \ref{rem:attain}.
Let us consider the remaining case. 
Assume that \eqref{eq:pop} has a global minimizer. 
Theorem \ref{theo:exact.pop} implies that Algorithm \ref{alg:POP} produces finitely many semidefinite programs whose smallest minimum value is exactly $h^\star$.
Hence the result follows.
\end{proof}

\begin{remark}
The paper (particularly Corollary \ref{coro:equi.pop.sdp}) provides a theoretical overview of why every general polynomial optimization problem can be solved with a sequence of semidefinite programs obtained recursively.
However, our method should be far from efficient in practice because of the vast bounds \eqref{eq:xi.KKT} and \eqref{eq:xi.FJ} in Algorithm \ref{alg:POP}.
This is unavoidable since we use symbolic computation to obtain the result.
Thus our method cannot be expected to be as efficient as approximation methods.
\end{remark}
\if{
\begin{remark}
We use the same assumption and notation as in Remark \ref{rem:symbolic}.
For each $k\in\N$, define
\begin{equation}\label{eq:equi.sos}
\begin{array}{rl}
\tilde \rho_k(\bar h)=\sup\limits_{\xi\in\R} &\xi\\
\text{s.t.}&h_0-\xi\in \Sigma_k[x]+I_k({\bar q_{\KKT}^U})[x]\,,\,U\in A\,,\\[5pt]
&h_0-\xi\in \Sigma_k[x]+I_k(p^U)[x]\,,\,U\in B\,,
\end{array}
\end{equation}
where $p^U$ is the vector of polynomials defining $U$, for $U\in B$.
Remark \ref{rem:symbolic} implies that $h_0-h^\star\ge 0$ on $V({\bar q_{\KKT}^U})$, for $U\in A$, and $h_0-h^\star\ge 0$ on $V(p^U)$, for $U\in B$.
Theorems \ref{theo:rep.KKT} and \ref{theo:rep.finite} yields 
\end{remark}
}\fi
\section{Equivalence of hyperpolic programs and semidefinite programs}
\label{sec:hyperbolic}
A polynomial $f$ in  $\R[x]$ is called \emph{hyperbolic} w.r.t. $e=(e_1,\dots,e_n) \in \R^n$ if $f(e)\ne 0$ and for all $a \in \R^n$, the univariate polynomial $t\mapsto f(t e - a)$ has only real roots.

Given a polynomial $f$ hyperbolic w.r.t. $e \in \R^n$, we define the set
\begin{equation}
\Lambda_+(f, e) := \{a\in \R^n\,:\, f(t e - a) = 0 \Rightarrow t \ge 0\}\,,
\end{equation}
which is called the hyperbolic cone of $f$ in direction $e$.

As shown by G{\aa}rding in \cite{gaarding1959inequality}, $\Lambda_+(f, e)$ is a basic semi-algebraic set.
We state in the following lemma a weaker result than this statement:
\begin{lemma}\label{lema:semi.hyper}
Let $f$ be a hyperbolic polynomial in $\R[x]$  w.r.t. $e \in \R^n$.
Then $\Lambda_+(f, e)$ is the projection of a basic semi-algebraic set defined by polynomials built from $e$ and the coefficients of $f$.
\end{lemma}
\begin{proof}
Let $t$ be a single variable. Set $d:=\deg(f)$ and $\varphi_a(t)=f(te-a)$. 
Then $\varphi_a$ is a polynomial in $t$ of degree at most $d$.
Denote by $\varphi_a^{(j)}$, $j=0,\dots,d$, the coefficient of $\varphi_a$ associated with monomial $t^j$, $j=0,\dots,d$, respectively.
Then each mapping $a\mapsto \varphi_a^{(j)}$ is a polynomial in $a$.
Let $a\in \Lambda_+(f, e)$.
Then identity polynomial $t$ is non-negative on the zero-dimensional algebraic variety $V(\varphi_a)$.
Assume $V(\varphi_a)=\{t_1,\dots,t_r\}$ for some $r\in\N$.
Then it holds that $t_j\ge 0$, $j=1,\dots,r$.
Since $\varphi_a$ has degree at most $d$, we get $r\le d$.
Letting
\begin{equation}
\arraycolsep=1.4pt\def\arraystretch{.5}
p_j(t):=\prod\limits_{\begin{array}{cc}
\scriptstyle i=1\\
\scriptstyle i\ne j
\end{array}}^r\frac{t-t_i}{t_j-t_i}\,,\,j=1,\dots,r\,,
\end{equation}
we obtain $p_j(t_i)=\delta_{ij}$. 
Setting $\sigma=\sum_{j=1}^r t_jp_j^2$ gives $\sigma\in\Sigma^2_{r-1}[t]$.
Moreover, $t-\sigma$ vanishes on $V(\varphi_a)$.
Since $r\le d$, $\sigma$ has degree at most $2(d-1)$.
It implies that both $\varphi_a$ and $\sigma$ have degrees at most $2d$.
Set $\xi:=\frac{1}{2}\times b(1,2d,2)$ and $s:=2\lfloor \xi/(2d)\rfloor$, where $b(\cdot)$ is defined as in \eqref{eq:bound.Krivine}.
By Lemma \ref{lem:pos2}, it holds that 
\begin{equation}\label{eq:suffi.nesces}
\exists \sigma\in\Sigma^2_{d-1}[t]\,:\,-(t-\sigma)^s \in \Sigma^2_\xi[t]+ I_\xi(\varphi_a)[t] 
\end{equation}
Note that $s$ and $\xi$ only depend on $d$.
From this, \eqref{eq:suffi.nesces} is equivalent to ``$t\ge 0$ on $V(\varphi_a)$''.
Moreover, \eqref{eq:suffi.nesces} can be written as
\begin{equation}\label{eq:rep.gram}
\begin{cases}
\exists G_j\succeq 0\,,\,\exists u\in \R^{2(\xi-d)+1}\,:\\
-(t-v_{d-1}^\top G_0v_{d-1})^s=v_{\xi}^\top G_1v_{\xi}+v_{2(\xi-d)}^\top u\varphi_a\,,
\end{cases}
\end{equation}
where $v_\eta$ is the vector of monomials in $t$ up to degree $\eta$.
It not hard to see that \eqref{eq:rep.gram} is equivalent to
\begin{equation}\label{eq:rep.gram2}
\begin{cases}
\exists G_0\in \R^{d\times d}\,,\,\exists G_1\in \R^{(\xi+1)\times (\xi+1)}\,,\,\exists u\in \R^{2(\xi-d)+1}\,:\\
G_j^{\top}=G_j\,,\,F(G_j)\ge 0\,,\,
H(G_0,G_1,u,\varphi_a^{(0)},\dots,\varphi_a^{(d)})=0\,,
\end{cases}
\end{equation}
where $F$ denotes the vector of principal minors of a matrix, and $H$ is a vector of polynomials.
Since $\varphi_a^{(j)}$ is a polynomial in $a$, \eqref{eq:rep.gram2} is equivalent to that $a$ is in the projection of a basic semi-algebraic set on the coordinates w.r.t. $a$, which yields the result.
\end{proof}
We close the paper by stating what would be a consequence of the generalized Lax conjecture (see, e.g., \cite[Remark 1.7]{amini2019spectrahedrality}):
\begin{corollary}\label{coro:cons.Lax.conjecture}
Let $f$ be a hyperbolic polynomial in $\R[x]$  w.r.t. $e \in \R^n$.
Consider the following hyperbolic program: 
\begin{equation}\label{eq:hyper.program}
\begin{array}{rl}
\inf\limits_{x \in \R^n}& c^\top x\\
\text{s.t.}& Ax=b\,,\,x\in \Lambda_+(f, e)\,,
\end{array}
\end{equation}
where $A\in\R^{m\times n}$, $b\in\R^m$ and $c\in\R^n$ are given.
Then there is a symbolic algorithm with input $A,b,c,f,e$ which produces a sequence of semidefinite programs with minimum values $(\zeta_r)_{r=1}^s$ such that \eqref{eq:hyper.program} is equivalent to the semidefinite program with smallest minimum value $\zeta_{\bar r}=\min\limits_{r=1,\dots,s}\zeta_r$.
\end{corollary}
\begin{proof}
Let $h^\star$ be the infimum value of problem \eqref{eq:hyper.program}.
Thanks to Lemma \ref{lema:semi.hyper}, \eqref{eq:hyper.program} can be written as polynomial optimization problem \eqref{eq:pop} with minimum value $h^\star$ (see Remark \ref{rem:convert}).
Note that the objective and constraint polynomials $h_0,\dots,h_l$ are built from $A,b,c,f,e$.
By Corollary \ref{coro:equi.pop.sdp}, there is a symbolic algorithm with input $h_0,\dots,h_l$ which produces a finite sequence of semidefinite programs such that \eqref{eq:pop} is equivalent to a semidefinite program of the smallest minimum value.
Hence the result follows.
\end{proof}
\begin{remark}
The framework of Corollary \ref{coro:cons.Lax.conjecture} provides a constructive way to obtain an equivalent semidefinite program for the hyperbolic program \eqref{eq:hyper.program} with given $A,b,c,f,e$.
However, it has yet to easily be seen that the result of Corollary \ref{coro:cons.Lax.conjecture} provides a solution to the generalized Lax conjecture.
It is because the equivalence of hyperbolic and semidefinite programs might not directly imply the equality of their feasible sets.
\end{remark}




\paragraph{Acknowledgements.} 
The authors would like to thank Bernard Mourrain, Markus Schweighofer, and Claus Scheiderer for their valuable discussions about this paper.

The first author was supported by the funding from ANITI.
The second author was supported by the Tremplin ERC Stg Grant ANR-18-ERC2-0004-01 (T-COPS project) and by the FMJH Program PGMO (EPICS project) and  EDF, Thales, Orange et Criteo.
This work has benefited from the European Union's Horizon 2020 research and innovation programme under the Marie Sklodowska-Curie Actions, grant agreement 813211 (POEMA) as well as from the AI Interdisciplinary Institute ANITI funding, through the French ``Investing for the Future PIA3'' program under the Grant agreement n$^{\circ}$ANR-19-PI3A-0004.

\bibliographystyle{abbrv}

\end{document}